\documentclass[12pt]{amsart}
\usepackage[left=3cm,top=2cm,right=3cm]{geometry}
\usepackage{fancyhdr}
\usepackage{bbm}
\usepackage{mathrsfs}
\usepackage[usenames,dvipsnames,svgnames,table]{xcolor}
\usepackage[pdftex,pagebackref,colorlinks=true,linkcolor=BrickRed,citecolor=OliveGreen,pdfstartview=FitH]{hyperref}
\usepackage{amssymb}

\usepackage{graphicx}
\usepackage{hyperref}
\usepackage{float}
\usepackage[normalem]{ulem}
\usepackage{verbatim}

\newtheorem{thm}{Theorem}[section]

\newtheorem{coro}[thm]{Corollary}
\newtheorem{lem}[thm]{Lemma}
\newtheorem{prop}[thm]{Proposition}

\theoremstyle{definition}

\newtheorem{eg}[thm]{Example}
\newtheorem{hyp}[thm]{Hypothesis}

\theoremstyle{remark}

\newtheorem{remk}[thm]{Remark}

\newcommand{\Rbb}{ {\mathbb R}}

\newcommand{\Cbb}{ {\mathbb C}}

\newcommand{\Ig}{I_{\gamma}}
\newcommand{\Igplus}{I_{\gamma+1}}
\newcommand{\Igk}{I_{\gamma_k}}

\newcommand{\Pg}{P_{\gamma}}

\newcommand{\Pgk}{P_{\gamma_k}}

\newcommand{\cj}[1]{\textcolor{blue}{(#1)}}

\title{Hot spots on cones and warped product manifolds}
\author{Lawford Hatcher}

\begin{document}

\begin{abstract}
    We study extrema of solutions to the heat equation (i.e. hot spots) on a class of warped product manifolds of the form $([0,L]\times M,dr^2+f(r)^2h)$ where $(M,h)$ is a closed Riemannian manifold. We prove that, under certain conditions on the warping function $f$, the statement of Rauch's hot spots conjecture holds for the corresponding warped product. We then go on to study the long-time behavior of hot spots on infinite cones over closed Riemannian manifolds. In this case, under appropriate hypotheses on the initial condition, there are four possible long-time behaviors depending only on the spectral gap of the fiber $(M,h)$. 
\end{abstract}

\maketitle

\section{Introduction}\label{intro}

Let $(M,h)$ be a connected, smooth, closed $(n-1)$-dimensional Riemannian manifold with $n\geq 2$. Let $I=[0,L]$ be a compact interval, and let $f:I\to (0,\infty)$ be a smooth function. Define a manifold with boundary by $\Omega:=I\times M$. We use the variable $r$ for points in $I$ and $x$ for points in $M$. Endow $\Omega$ with the Riemannian metric 
\begin{equation}\label{wmetric}
    g(r,x):=dr^2+f(r)^2h(x).
\end{equation} The resulting Riemannian manifold with boundary $(\Omega,g)$ is an example of a warped product manifold. We refer to the function $f$ as the \textit{warping function} of $(\Omega,g)$. Let \[0=\mu_1<\mu_2\leq \mu_3\leq \dots\] denote the eigenvalues of the Neumann Laplace-Beltrami operator on $(\Omega,g)$. The first goal of this paper is to describe the set of local extrema of the eigenfunctions corresponding to $\mu_2$.
\begin{thm}\label{compactneumannthm}
    Let $u$ denote a second Neumann eigenfunction of the Laplace-Beltrami operator on $(\Omega,g)$. If either of the following conditions holds, then every local extremum of $u$ is contained in $\partial\Omega$:
    \begin{enumerate}
        \item\label{thmcase1} $u$ is radial (i.e. depends only on the variable $r$).
        \item\label{thmcase2} The warping function $f$ is non-constant, (weakly) monotonic, and $\mu_2$ is a simple eigenvalue (i.e. $\mu_2<\mu_3$).
    \end{enumerate}
\end{thm}
\begin{remk}
    In Case (\ref{thmcase1}) of Theorem \ref{compactneumannthm}, we in fact prove the stronger statement that $u$ has no interior critical points. A point at which an eigenfunction and its gradient both vanish is known as a \textit{nodal critical point}. In Case (\ref{thmcase2}) of Theorem \ref{compactneumannthm}, we prove that if $u$ has an interior critical point, then it is a nodal critical point. In this same case, if $u$ is not radial, then it has a nodal critical point if and only if the fiber $(M,h)$ has a second eigenfunction with a nodal critical point. A celebrated result of Uhlenbeck \cite{uhlenbeck} states that, for a generically chosen metric $h$ on $M$, the pair $(M,h)$ has no eigenfunction with a nodal critical point.
\end{remk}
\indent Theorem \ref{compactneumannthm} states that Rauch's hot spots conjecture holds when $(\Omega,g)$ satisfies either (\ref{thmcase1}) or (\ref{thmcase2}). The original form of this conjecture states that on a bounded planar domain with generically chosen initial conditions,\footnote{The precise genericity condition is that the initial condition is not $L^2$-orthogonal to the second Neumann eigenspace.} the extrema (i.e. hot spots) of solutions to the Neumann heat equation tend toward the boundary of the domain as time tends to infinity. Since these extrema tend toward extrema of some second Neumann eigenfunction (see the exposition in \cite{rauch, banuelosburdzy}), Theorem \ref{compactneumannthm} implies this conjecture for these manifolds.\\
\indent Since the conjecture was made, the topic of hot spots has been explored on Riemannian manifolds and unbounded domains. See, e.g., \cite{chavelkarp, jimbo, freitas, ishige, kre, me4}. Of course, in this broader context, the statement of the hot spots conjecture often fails to hold (see Freitas's result \cite{freitas}) and does not even make sense for manifolds without boundary. \\
\indent In Section \ref{counterex} below, we show that the hypotheses that the warping function $f$ be non-constant and monotonic in Case (\ref{thmcase2}) of Theorem \ref{compactneumannthm} are both necessary via counterexamples. We note that a result of Kawohl \cite{kawohl} (see also the discussion in \cite{banuelosburdzy}) proves that a weakened version of the hot spots conjecture holds for Euclidean cylinders, which is just the case of a constant warping function. However, what Kawohl proves is that $u$ attains its extreme values on the boundary and not that it attains them \textit{only} on the boundary.\\
\indent It is unclear whether the simplicity hypothesis in Case (\ref{thmcase2}) of Theorem \ref{compactneumannthm} is necessary. As we describe in Example \ref{simplicitycounterex} below, there are two means by which the warped product $(\Omega,g)$ could have a multiple second Neumann eigenvalue. In one of these cases, we show that a second Neumann eigenfunction has no interior hot spots. In the other case, we construct second Neumann eigenfunctions with interior (non-nodal) critical points. However, it is currently unclear to us whether these critical points can be hot spots. We conjecture that they cannot. \\
\indent In recent years, researchers have considered analogous questions with mixed Dirichlet-Neumann eigenfunctions. See, e.g., \cite{me1}, \cite{liyao}, \cite{aldeghirohleder}, \cite{me2}, \cite{me3}, \cite{kennedyrohleder}. In this paper, we also consider the mixed eigenvalue problem in which Dirichlet boundary conditions are imposed on $\{0\}\times M$ and Neumann boundary conditions are imposed on $\{L\}\times M$. The first eigenvalue $\lambda_1$ of the Laplace-Beltrami operator on $\Omega$ with these boundary conditions is strictly positive, simple, and has non-constant corresponding eigenfunctions. A solution to the heat equation with these same boundary conditions and with a generically chosen initial condition has extrema tending toward the extrema of such an eigenfunction. In this case, the hypotheses are significantly simpler. 
\begin{thm}\label{compactmixedthm}
    Let $u$ denote a first eigenfunction of the Laplace-Beltrami operator on $(\Omega,g)$ with the boundary conditions described above. Then the critical points (and, therefore, extrema) of $u$ are contained in $\{L\}\times M$.
\end{thm}
We next turn our attention to the study of hot spots on a non-compact metric space known as the \textit{cone over} $(M,h)$, which is the metric completion of a certain singular warped product over $M$. In this case, we let $I$ equal the non-compact interval $(0,\infty)$, and we take $f(r)=r$ as the warping function to define the metric $g(r,x)$ as in (\ref{wmetric}). The metric completion of $(\Omega,g)$ with respect to $g$ is isometric to the union of $(\Omega,g)$ with the \textit{cone point}, which we call $p$. We denote this metric completion by $C(M)$ and refer to it as the cone over $(M,h)$. Let $d(\cdot,\cdot)$ denote the extension of the Riemannian distance on $(0,\infty)\times M$ to $C(M)$. One can define a self-adjoint Laplacian $\Delta_g$ on $C(M)$ via the Friedrichs extension (see \cite{cheeger, mooers}), and this Laplacian gives rise to a heat equation on $C(M)$:
\begin{equation}\label{heateqn}
    \begin{cases}
        \partial_t u=\Delta_gu\\
        u(r,x,0)=\phi(r,x),
    \end{cases}    
\end{equation}

Since the natural embedding of the Sobolev space $H^1(C(M))$ into $L^2(C(M))$ is not compact, the Laplacian $\Delta_g$ does not have discrete spectrum, and the hot spots of the solutions of the heat equation cannot be studied using eigenfunctions. We instead study the hot spots by analyzing the heat kernel on $C(M)$. Given a solution $u$ of the heat equation (\ref{heateqn}) on $C(M)$ and for each $t>0$, let \[H(t):=\Big\{(r,x)\in C(M)\mid u(r,x,t)=\max_{(s,y)}u(s,y,t)\Big\}.\]
We determine the long-time location of the set $H(t)$ under the following two assumptions on the initial condition $\phi$:
\begin{hyp}\label{phihyp}
    Suppose that the initial condition $\phi\in L^2(C(M))$ is compactly supported and $\displaystyle\int_{C(M)}\phi dV_g>0$.
\end{hyp}
\indent Let $V$ denote the set of functions satisfying Hypothesis \ref{phihyp} endowed with the subspace topology inherited from $L^2(C(M))$. In Theorem \ref{mainthmcones} below, we make an assumption (namely, Hypothesis \ref{hyp2} below) that holds in an open dense subset of $V$. 
\begin{thm}\label{mainthmcones}
    Let $(M,h)$ be a closed Riemannian manifold of dimension $n-1\geq 1$. Let $\nu_2$ be the second eigenvalue of the Laplace-Beltrami operator on $(M,h)$. If $u$ solves the heat equation (\ref{heateqn}) with initial condition $\phi$ satisfying Hypothesis \ref{phihyp} and Hypothesis \ref{hyp2} below, then
    \begin{enumerate}
        \item If $\nu_2\geq 2n$, then there exists $T>0$ such that for all $t\geq T$, \[H(t)=\{p\}.\]
        \item\label{thm3case2} If $n-1<\nu_2<2n$, then  \[\lim_{t\to\infty}\sup\{r\mid (r,x)\in H(t)\}=0.\]
        \item\label{thm3case3} If $\nu_2=n-1$, then there exists a set $H_{\infty}(\phi)\subseteq C(M)$, defined below, such that \[\lim_{t\to\infty}\sup_{a\in H(t)}d(a,H_{\infty}(\phi))=0.\]
        \item\label{thm3case4} If $\nu_2<n-1$, then  \[\lim_{t\to\infty}\inf\{r\mid (r,x)\in H(t)\}=\infty.\]
    \end{enumerate}
\end{thm}

\begin{remk}
    In Cases (\ref{thm3case2}), (\ref{thm3case3}), and (\ref{thm3case4}) of Theorem \ref{mainthmcones}, we in fact provide significantly more precise information on the location of $H(t)$ for large times. In fact, Theorem \ref{mainthmcones} follows immediately from the more detailed Theorem \ref{moreprecise} in Section \ref{mainproofsection} below. In particular, one can see in Theorem \ref{moreprecise} that Case (\ref{thm3case2}) in Theorem \ref{mainthmcones} is sharp in that the hot spots never quite reach the cone point as $t\to\infty$.
\end{remk}

\indent We now define the set $H_{\infty}(\phi)$ in Case (\ref{thm3case3}) of Theorem \ref{mainthmcones} and state Hypothesis \ref{hyp2} along the way. Denote the remaining eigenvalues of the Laplace-Beltrami operator on the fiber $(M,h)$ by $0=\nu_1<\nu_2\leq \nu_3\leq \dots$, and let $\{v_k(x)\}$ denote a corresponding orthonormal basis of eigenfunctions. For each $k\geq 1$, let \[\gamma_k:=\sqrt{\frac14(n-2)^2+\nu_k}.\] Suppose that \[\nu_2=...=\nu_{K-1}<\nu_{K}\] for some $K\geq 3$. Let 
\begin{equation}
    F(x):=\sum_{k=2}^{K-1}\Big(\int_{C(M)}s^{\gamma_k-n/2+1}\cdot \phi(s,y)\cdot v_k(y)dV_g(s,y)\Big)v_k(x)\in L^2(M,h).
\end{equation}
The final hypothesis in Theorem \ref{mainthmcones} is
\begin{hyp}\label{hyp2}
    Suppose that the function $F(x)$ is not identically equal to $0$. 
\end{hyp}
Let $A_{\infty}$ be the set of points in $M$ at which $F$ takes its maximum value. Define \[r_{\infty}:=\frac1n\cdot\text{Vol}(M,h)\cdot\frac{\displaystyle\max_{x\in M}F(x)}{\displaystyle\int_{C(M)}\phi(s,y)dV_g(s,y)}.\] Then the set $H_{\infty}(\phi)$ is defined by \begin{equation}
    H_{\infty}(\phi):=\begin{cases}
        \{r_{\infty}\}\times A_{\infty}\;\;&\text{if}\;\;r_{\infty}>0\\\\
        \{p\}&\text{if}\;\; r_{\infty}=0.
    \end{cases}
\end{equation}

\begin{remk}
    The Euclidean space $\Rbb^n$ is isometric to the cone over the unit $(n-1)$-sphere. On $\Rbb^n$, Chavel and Karp prove in \cite{chavelkarp} that, given a non-negative compactly supported initial condition, the hot spots tend toward the center of mass of the initial condition. The second eigenvalue of the unit $(n-1)$-sphere is equal to $n-1$, and the next eigenvalue not equal to $n-1$ is equal to $2n$. Assuming that Hypothesis \ref{hyp2} holds, one can check that Theorem \ref{mainthmcones} implies Chavel and Karp's result. In fact, using Lemma \ref{gaussianest}, Lemma \ref{wderivative}, and Proposition \ref{tailderivative} below, Hypothesis \ref{hyp2} can even be removed in this case. 
\end{remk}
\begin{remk}
    The proof of Theorem \ref{mainthmcones} is inspired by Ishige's proof of his main result in \cite{ishige}, in which he determines the location of hot spots on $\Rbb^n$ minus the unit ball with Neumann boundary conditions. In his paper, Ishige implicitly makes use of a warped product decomposition of this space and analyzes the functions obtained by separation of variables. This warped product is the same as the cone over the unit sphere minus the unit ball. Based on Ishige's work, we conjecture that a statement similar to Theorem \ref{mainthmcones} holds for manifolds with boundary of the form $([1,\infty)\times M,dr^2+r^2h)$ with Neumann boundary conditions on $\{1\}\times M$. 
\end{remk}

We now outline the remainder of the paper. In Section \ref{prelim}, we review some basic facts about the geometry and spectrum of $(\Omega,g)$ with respect to the warping function $f$ and fiber $(M,h)$. Using these facts, we prove Theorems \ref{compactneumannthm} and \ref{compactmixedthm} in Section \ref{compactproofs}. Following these proofs, we show that the hypotheses in Theorem \ref{compactneumannthm} are necessary via some counterexamples in Section \ref{counterex}. We then begin our study of hot spots on cones. We begin with Section \ref{besselsection}, in which we recall some facts about modified Bessel functions. These functions are a key ingredient in the explicit heat kernel for $C(M)$, which we describe in Section \ref{heatkernelsection}. In Sections \ref{limitingsection} and \ref{tailsection}, we analyze different parts of the expression for the heat kernel to understand the long-term behavior of solutions of the heat equation. We then combine these analyses to prove Theorem \ref{mainthmcones} in Section \ref{mainproofsection}.

\section{Preliminaries on warped product manifolds}\label{prelim}

Here we review some facts about the geometry and spectra of warped product manifolds. Let $(M,h)$, $f:I\to (0,\infty)$, and $(\Omega,g)$ be as described in Section \ref{intro}. Throughout the paper we let $x$ denote a point in $M$, and we let $dx$ denote the volume density on $(M,h)$. A simple computation shows that the volume density on $(\Omega,g)$ is given by \[dV_g(r,x)=f(r)^{n-1}drdx.\] Moreover, if $\Delta_h$ is the Laplace-Beltrami operator on $(M,h)$, then the Laplace-Beltrami operator on $(\Omega,g)$ is given by \[\Delta_g=\partial_r^2+(n-1)\cdot\frac{f'(r)}{f(r)}\cdot \partial_r+\frac{1}{f(r)^2}\cdot \Delta_h.\]

\indent The structure of $\Omega$ allows us to perform separation of variables on the eigenvalue problems on $(\Omega,g)$. Let $\nu_k$ denote the eigenvalues of $(M,h)$ and $v_k(x)$ an orthonormal basis of corresponding eigenfunctions. Then the Neumann eigenvalues of $(\Omega,g)$ come in a two-parameter family $\mu_{j,k}$ with eigenfunctions $u_{j,k}(r,x)=w_{j,k}(r)v_k(x)$ where $j$ and $k$ each run over the positive integers. The functions $w_{j,k}$ satisfy the ordinary differential equations 
\begin{equation}\label{radialevals}
    \begin{cases}
    w_{j,k}''(r)+(n-1)\cdot\displaystyle\frac{f'(r)}{f(r)}\cdot w_{j,k}'(r)-\displaystyle\frac{\nu_k}{f(r)^2}\cdot w_{j,k}(r)=-\mu_{j,k}w_{j,k}(r)\\ \\
    w_{j,k}'(0)=w_{j,k}'(L)=0.
\end{cases}
\end{equation}
The eigenvalues of the mixed Dirichlet-Neumann problem mentioned in Section \ref{intro} can also be treated via separation of variables. Abusing notation slightly, we use the same variables for the eigenvalues and eigenfunctions of the mixed problem as we do for the Neumann problem, and the top line of Equation (\ref{radialevals}) also holds for the radial components of the corresponding eigenfunctions. The only difference in the bottom line is that we impose the boundary condition $w_{j,k}(0)=0$ rather than $w_{j,k}'(0)=0$.\\
\indent It will be useful in the following sections to understand the relationship between the $\nu_k$ and the $\mu_{j,k}$. In what follows, we suppress the subscript $k$ from the notation for both $\nu$ and $\mu$.

\begin{lem}\label{decinnu}
    For either the Neumann or mixed problem, the eigenvalues $\mu_j$ of (\ref{radialevals}) are each increasing in $\nu$. For the Neumann problem, as $\nu$ decreases to $0$, $\mu_1$ approaches $0$. 
\end{lem}
\begin{proof}
    We first consider the Neumann case. For each $\nu\geq 0$, define a quadratic form $q_{\nu}:H^1([0,L])\to \Cbb$ by \[q_{\nu}(u)=\int_0^L|u'(r)|^2\cdot f(r)^{n-1}dr+\int_0^L\frac{\nu}{f(r)^2}\cdot |u(r)|^2\cdot f(r)^{n-1}dr.\] Then the variational characterization of the eigenvalues of (\ref{radialevals}) is (for details, see, e.g., Theorem 5.15 of Borthwick's text \cite{borthwick})
    \begin{equation}\label{variational}
        \mu_j=\min_{\substack{V\subseteq H^1([0,L])\\\text{dim}(V)=j}}\max\Bigg\{q_{\nu}(u)\mid u\in V, \int_0^L|u(r)|^2\cdot f(r)^{n-1}dr=1\Bigg\}.
    \end{equation}
    Since for a fixed $u\in H^1([0,L])$, the map $\nu\mapsto q_{\nu}(u)$ is increasing, it follows that the eigenvalues $\mu_j$ are also increasing in $\nu$. The mixed case of the first statement is similar, the only exception being that the variational characterization for the eigenvalues is defined using the Sobolev space of $H^1$ functions that vanish at $r=0$ as test functions. \\
    \indent For the second statement, we use the constant function \[u(r)\equiv\frac{1}{\sqrt{\int_0^Lf(r)^{n-1}dr}}\] as a test function. By the variational characterization (\ref{variational}), this gives \[\mu_1\leq q_{\nu}(u)=\nu\cdot \Bigg(\frac{\int_0^Lf(r)^{n-3}dr}{\int_0^Lf(r)^{n-1}dr}\Bigg),\] implying that the first eigenvalue approaches $0$ as $\nu\to0$.
\end{proof}

Using that $\mu_{1,1}=0$ for the Neumann problem and that $\mu_{j,k}$ are each increasing in both $j$ and $k$, one concludes the following:

\begin{coro}\label{secondneumann}
    The first non-zero Neumann eigenvalue of the Laplace-Beltrami operator on $(\Omega,g)$ is equal to either $\mu_{1,2}$ or $\mu_{2,1}$. Using the mixed version of (\ref{radialevals}), the first mixed eigenvalue of $(\Omega,g)$ is equal to $\mu_{1,1}$.
\end{coro}

\section{Proofs of theorems \ref{compactneumannthm} and \ref{compactmixedthm}}\label{compactproofs}

To prove Theorems \ref{compactneumannthm} and \ref{compactmixedthm}, it suffices to prove that in each case, the derivative of the radial component of the appropriate eigenfunction (see Section \ref{prelim}) does not vanish in $(0,L)$. We begin with Theorem \ref{compactneumannthm}. For ease of exposition, we temporarily suppress the spectral indices introduced in Section \ref{prelim}. Let $\nu\geq 0$ and $\mu>0$. Suppose then that $w:I\to \Rbb$ satisfies
\begin{equation}\label{weqn}
    \begin{cases}
    w''+(n-1)\cdot\displaystyle\frac{f'}{f}\cdot w'-\displaystyle\frac{\nu}{f^2}\cdot w=-\mu w\\ \\
    w'(0)=w'(L)=0.
    \end{cases}
\end{equation}
Define a function \[a(r):=f(r)^{n-1}\cdot\Big(\frac{\nu}{f(r)^2}-\mu\Big)\cdot w(r).\]
Then the top line of Equation (\ref{weqn}) can be rewritten as 
\begin{equation*}
    (f^{n-1}\cdot w')'=a(r).
\end{equation*}
Integrating both sides of this equation and using the boundary condition $w'(0)=0$ yields 
\begin{equation}\label{weqn2}
    w'(r)=\frac{1}{f(r)^{n-1}}\int_0^ra(s)ds.
\end{equation}

\begin{lem}\label{ahasonezero}
    Replacing $w$ with $-w$ if necessary, suppose that $w(L)>0$. Suppose that one of the following holds:
    \begin{enumerate}
        \item $\nu=0$ and $w$ has exactly one zero in the interval $I$, or
        \item $\nu>0$, $f$ is non-constant and non-decreasing, and $w>0$ on $I$.
    \end{enumerate}
    Then there exist $0<r_1\leq r_2<L$ such that 
    \begin{enumerate}
        \item $a(r)>0$ on $[0,r_1)$,
        \item $a(r)=0$ on $[r_1,r_2]$, and
        \item $a(r)<0$ on $(r_2,L]$.
    \end{enumerate}
\end{lem}
\begin{proof}
    If $\nu=0$ and $w$ has exactly one zero in $I$ with $w(L)>0$, then we may take $r_1=r_2$ equal to the unique zero of $w$.\\
    \indent Suppose then that the hypothesis (2) holds. In this case, the function $r\mapsto \frac{\nu}{f(r)^2}-\mu$ is non-constant and non-increasing. Since $w>0$ and $f>0$ on $I$, it follows that $a^{-1}(\{0\})$ is a closed interval. If $a(0)\leq 0$, then $a\leq 0$ on all of $I$ and is strictly negative on some sub-interval of $I$. In this case, Equation \ref{weqn2} implies that $w'(L)<0$, contradicting the boundary condition $w'(L)=0$. Thus, $a(0)>0$. Similarly, if $a(L)\geq 0$, then $a\geq 0$ on all of $I$, and $w'(L)>0$, also contradicting the boundary condition. Hence, $a(L)<0$.
\end{proof}

\begin{lem}\label{almostneumannthm}
    Suppose that $\nu$ and $w$ satisfy either of the hypotheses of Lemma \ref{ahasonezero}. Then $w'(r)>0$ for all $r\in (0,L)$.
\end{lem}
\begin{proof}
    By Lemma \ref{ahasonezero} and Equation (\ref{weqn2}), $w'(r)>0$ for $r\in(0,r_2]$ with $r_2$ as in Lemma \ref{ahasonezero}. Suppose toward a contradiction that $w'(r_0)=0$ for some $r_0\in (0,L)$. Then $r_0>r_2$, and it follows from Lemma \ref{ahasonezero} that $a(r)<0$ on $[r_0,L]$. Using Equation (\ref{weqn2}), we get a contradiction to the boundary condition $w'(L)=0$: 
    \begin{align*}
        w'(L)&=\frac{1}{f(L)^{n-1}}\cdot \int_0^L
    a(s)ds\\
            &=\frac{1}{f(L)^{n-1}}\cdot\Bigg(\int_0^{r_0}a(s)ds+\int_{r_0}^La(s)ds\Bigg)\\
            &=\frac{1}{f(L)^{n-1}}\cdot\int_{r_0}^L a(s)ds\\ \\
            &<0.
    \end{align*}
\end{proof}

\begin{proof}[Proof of Theorem \ref{compactneumannthm}]
    Let $u$ denote a second Neumann eigenfunction of $(\Omega,g)$. By Corollary \ref{secondneumann} and the simplicity of the second Neumann eigenvalue, the function $u$ is a constant multiple of either \[(r,x)\mapsto w_{2,1}(r)\] or \[(r,x)\mapsto w_{1,2}(r)v_2(x).\] 
    \indent In the first case of the theorem statement, the function $w_{2,1}$ is radial and satisfies Equation (\ref{weqn}) with $\nu=0$. Since $w_{2,1}$ corresponds to the first eigenvalue of a Sturm-Liouville eigenvalue problem, we may suppose that it is strictly positive, and the conclusion follows immediately from Lemma \ref{almostneumannthm}.\\
    \indent In the second case of the theorem statement, if the eigenfunction is not radial, then it is of the form $w_{1,2}(r)v_2(x)$, and we suppose without loss of generality that the warping function $f$ is non-decreasing. By standard results in Sturm-Liouville theory, $w_{2,1}$ has exactly one zero in $I$. By Lemma \ref{almostneumannthm}, $w_{2,1}$ does not have a critical point in the interval $(0,L)$. We conclude that for $r_0\in (0,L)$, a point $(r_0,x_0)$ is a critical point of $w_{1,2}(r)v_2(x)$ if and only if $x_0$ is a nodal critical point of $v_2$. In this case, $(r_0,x_0)$ is a nodal critical point of $w_{1,2}(r)v_2(x)$.
\end{proof}

The proof of Theorem \ref{compactmixedthm} is simpler since, by Corollary \ref{secondneumann}, a first mixed eigenfunction of $\Omega$ is a scalar multiple of \[(r,x)\mapsto w_{1,1}(r).\] Similarly to the Neumann case, the function $w_{1,1}$ satisfies the integral equation
\begin{equation}\label{wmixedeqn}
    \begin{cases}
        f(r)^{n-1}\cdot w_{1,1}'(r)=f(0)^{n-1}\cdot w_{1,1}'(0)-\mu_{1,1}\displaystyle\int_{0}^rf(s)^{n-1}\cdot w_{1,1}(s)ds\\
        w_{1,1}(0)=0\\
        w_{1,1}'(L)=0.
    \end{cases}
\end{equation}

\begin{proof}[Proof of Theorem \ref{compactmixedthm}]
    Multiplying by $-1$ if necessary, we may suppose without loss of generality that $w_{1,1}\geq 0$ on $[0,L]$. By the Hopf lemma and the Dirichlet boundary condition at $r=0$, we have $w_{1,1}'(0)>0$. Since $\mu_{1,1}>0$ and $f>0$, it follows from Equation (\ref{wmixedeqn}) that $f(r)^{n-1}\cdot w_{1,1}'$ is strictly decreasing. Since $w_{1,1}'(L)=0$, the function $w_{1,1}'(r)$ cannot have any other zeros in the interval $I$, and the theorem follows. 
\end{proof}

\section{Some counterexamples}\label{counterex}
We now present a sequence of examples showing that the hypotheses in Theorem \ref{compactneumannthm} that the warping function be non-constant and monotonic are each necessary. The simplest counterexample is given by the trivial product of a sufficiently short interval with any compact manifold, described in the next paragraph. However, this simple construction still produces second Neumann eigenfunctions that achieve their extreme values on the boundary of the manifold. We next construct manifolds that have second Neumann eigenfunctions that achieve their extreme values only on the interior. A slight modification of this construction gives second Neumann eigenfunctions with interior saddle critical points. We end the section by describing the simplicity issue in Theorem \ref{compactneumannthm} in more detail. 

\begin{eg}[Trivial products]
    Let $(M,h)$ be a closed manifold, and let $\Omega=[0,L]\times M$ endowed with the Riemannian metric $dr^2+h(x)$ where $r\in [0,L]$ and $x\in M$. Let $0=\nu_1<\nu_2\leq \nu_3\leq\dots$ denote the eigenvalues of $(M,h)$ and $\{v_k\}$ the eigenfunctions. By separating variables, we see that the eigenvalues of the Laplace-Beltrami operator on $(\Omega,g)$ with Neumann boundary conditions are given by the two-parameter family
    \[\mu_{j,k}=\frac{(j-1)^2\pi^2}{L^2}+\nu_k,\;\;j,k\geq 1\]
    with corresponding (non-normalized) eigenfunctions given by $$u_{j,k}(r,x)=\cos\Big(\frac{(j-1)\pi}{L}\cdot r\Big)\cdot v_k(x).$$ Thus, for $L$ sufficiently small, a second Neumann eigenfunction is given by the function $u_{1,2}(r,x)$. If $x_0\in M$ is a point at which $v_2$ attains a maximum (resp. minimum), then $u_{1,2}$ attains its maximum (resp. minimum) on the set $[0,L]\times \{x_0\}$.
\end{eg}

\begin{eg}[Warped products with symmetric warping functions]
    For ease of exposition, we now consider an interval of the form $I=[-L,L]$. Let $f:I\to (0,\infty)$ be a non-constant, even, smooth function (i.e. $f(-x)=f(x)$ for each $x\in I$). Further suppose that $f$ is monotonic on $[0,L]$. For a positive parameter $c>0$, consider the warped product manifold \[(\Omega,g_c):=\big(I\times M,dr^2+c^2f(r)^2h(x)\big).\] Using the notation from Section \ref{prelim} and applying Lemma \ref{decinnu} with $\nu=\nu_2/c^2$, if $c$ is sufficiently large, then a second Neumann eigenfunction of $\Omega$ is a scalar multiple of the function $u_{1,2}(r,x)=w_{1,2}(r)\cdot v_2(x)$ (to see this, note that the eigenvalue $\mu_{2,1}$ is independent of the value of $c$).\\
    \indent The map $(r,x)\mapsto (-r,x)$ is an isometry of $(\Omega,g_c)$. Since the Laplace-Beltrami operator commutes with pullback operators given by isometries, $w_{1,2}$ is either even or odd. Since this function does not change signs in the interval $I$, it must therefore be even. Thus, $w_{1,2}'(0)=0$. If $x_0\in M$ is a non-nodal critical point of $v_2$, then $(0,x_0)$ is an interior, non-nodal critical point of the eigenfunction $u_{1,2}$.\\
    \indent Now suppose that $f$ is increasing on the interval $[-L,0]$ (and thus decreasing on $[0,L]$), and suppose without loss of generality that $w_{1,2}>0$ on $I$. By Lemma \ref{almostneumannthm}, $w_{1,2}$ is increasing on $[-L,0]$. Thus, $w_{1,2}$ attains its maximum uniquely at $r=0$. If $x_0$ is a maximum of $v_2$, then we conclude that $u_{1,2}$ attains its maximum at $(r,x)=(0,x_0)$.\\
    \indent By a similar argument, if $f$ is decreasing on $[-L,0]$, then $w_{1,2}$ attains its minimum uniquely at $r=0$. In this case, if $x_0$ is an extremum of $v_2$, then $(r,x)=(0,x_0)$ is a saddle point of the eigenfunction $u_{1,2}$.
\end{eg}

\begin{eg}[Warped products with a multiple second Neumann eigenvalue]\label{simplicitycounterex}
    By Corollary \ref{secondneumann}, if the warped product $(\Omega,g)$ has a multiple second Neumann eigenvalue, then this arises from at least one of the following alternatives:
    \begin{enumerate}
        \item\label{casec1} The manifold $(M,g)$ has a multiple second eigenvalue: $\nu_2=...=\nu_{K-1}<\nu_{K}$ and $\mu_{1,2}=...=\mu_{1,K-1}$ is the second Neumann eigenvalue of $(\Omega,g)$.
        \item\label{casec2} The second Neumann eigenvalue of $(\Omega,g)$ is equal to $\mu_{1,2}=\mu_{2,1}$.
    \end{enumerate}
    
    \indent If Alternative (\ref{casec1}) holds but Alternative (\ref{casec2}) does not hold while the other hypotheses of Theorem \ref{compactneumannthm} hold, then the conclusion of the theorem still holds. Indeed, in this case, $\mu_{1,2}=...=\mu_{1,K-1}<\mu_{2,1}$, so the radial components of the eigenfunctions corresponding to the second Neumann eigenvalue are equal:
    \[w_{1,2}(r)=...=w_{1,K-1}(r).\] A second Neumann eigenfunction is then a linear combination of the form
    \[u(r,x)=\sum_{k=2}^{K-1}a_ku_{1,k}(r,x)=w_{1,2}(r)\cdot \sum_{k=2}^{K-1}a_kv_k(x).\] Then Lemma \ref{almostneumannthm} implies that $\partial_ru$ does not vanish in $(0,L)\times M$ unless \[\sum_{k=1}^{K-1}a_kv_k(x_0)=0\] for some $x_0\in M$. Thus, each interior critical point of $u$ is a nodal critical point.\\
    \indent On the other hand, when Alternative (\ref{casec2}) holds, there exist second Neumann eigenfunctions with interior critical points that may or may not be nodal. In fact, we claim that in this case, for any $r_0\in (0,L)$, there exists $x_0\in M$ and a second Neumann eigenfunction $u$ for which $(r_0,x_0)$ is a critical point. Indeed, let $x_0\in M$ be a point at which $\nabla_hv_2(x_0)=0$ (where $\nabla_h$ denotes the Riemannian gradient on $M$) and $v_2(x_0)\neq 0$. Let \[b:=-\frac{w_{2,1}'(r_0)v_1(x_0)}{w_{1,2}'(r_0)v_2(x_0)}.\] Since $v_2(x_0)\neq 0$ and $w_{1,2}'(r_0)\neq 0$ by Lemma \ref{almostneumannthm}, this value is well defined. Then one can check, using that $v_1$ is a constant function, that \[u(r,x):=w_{2,1}(r)\cdot v_1(x)+b\cdot w_{1,2}(r)\cdot v_2(x)\] is the desired eigenfunction. It is unclear to us, however, whether this critical point is an extremum.
\end{eg}

\section{Modified Bessel functions}\label{besselsection}
Theorem \ref{mainthmcones} follows from a direct analysis of the heat kernel on $C(M)$ (see Section \ref{heatkernelsection}), which can be expressed in terms of modified Bessel functions and the eigenfunctions of $M$. We review these functions and some of their relevant properties here.\\
\indent Let $\gamma>0$, and denote by $\Ig(z)$ be the modified Bessel function of the first kind of order $\gamma$. See, for instance, \cite{nist} for each of the properties that we list without proof. We make use of two equivalent expressions for $\Ig$: 
\begin{equation}\label{taylorseries}
    \Ig(z)=\Big(\frac{z}{2}\Big)^{\gamma}\cdot \sum_{j=0}^{\infty}\frac{z^{2j}}{4^j\cdot j!\cdot\Gamma(\gamma+j+1)},
\end{equation}
\begin{equation}\label{integralrep}
    \Ig(z)=\frac{(z/2)^\gamma}{\sqrt{\pi}\cdot\Gamma(\gamma+1/2)}\int_{-1}^1(1-\tau^2)^{\gamma-1/2}e^{z\tau}d\tau.
\end{equation}
One can show that for $\gamma'>\gamma$ and $z\geq 0$,
\begin{equation}\label{besselmono}
    I_{\gamma'}(z)\leq \Ig(z).
\end{equation}
The heat kernel for $C(M)$ is an infinite series involving a modified Bessel function in each term. To analyze this series, we use a simple asymptotic on $\Ig$ as $z\to0$ as well as a straightforward global estimate that we prove below. For the asymptotic, note that the power series on the right-hand side of Equation (\ref{taylorseries}) defines an entire function. We therefore have 
\begin{equation}\label{besselatzero}
    \Ig(z)=\frac{1}{2^{\gamma}\cdot\Gamma(\gamma+1)}\cdot z^{\gamma}\cdot\big(1+O(z^2)\big)\;\;\text{as}\;\;z\to0.
\end{equation}
where the $O(z^2)$ term is an entire function. An immediate consequence of this asymptotic is that for any $\gamma>0$, we have 
\begin{equation}\label{quotientofbessels}
    \Igplus(z)=\frac{1}{2\cdot (\gamma+1)}\cdot z\cdot \Ig(z)\cdot\big(1+O(z^2)\big)\;\;\text{as}\;\;z\to0.
\end{equation}
\indent We now prove the global estimate mentioned above:
\begin{lem}\label{estimatebessel}
    For any $\gamma>0$ and $z\geq 0$, 
    \[\Ig(z)\leq \frac{1}{\Gamma(\gamma)}\cdot z^{\gamma}\cdot e^z.\]
\end{lem}
\begin{proof}
    We use the integral representation (\ref{integralrep}). We first bound the integral in this representation from above:
    \begin{align*}
        \int_{-1}^1(1-\tau^2)^{\gamma-1/2}e^{z\tau}d\tau&\leq 2e^z\int_0^1(1-\tau^2)^{\gamma-1/2}d\tau\\
        &\leq 2e^z\int_0^1\tau^{\gamma-1/2}\cdot(2-\tau)^{\gamma-1/2}d\tau\\
        &\leq 2^{\gamma+1/2}e^z\int_0^1\tau^{\gamma-1/2}d\tau\\
        &=\frac{2^{\gamma+1/2}e^z}{\gamma+1/2}.
    \end{align*}
    Plugging this into the integral representation of $\Ig$ gives
    \begin{align*}
        \Ig(z)&=\frac{(z/2)^\gamma}{\sqrt{\pi}\cdot\Gamma(\gamma+1/2)}\int_{-1}^1(1-\tau^2)^{\gamma-1/2}e^{z\tau}d\tau\\
        &\leq \sqrt{\frac{2}{\pi}}\cdot\frac{1}{\Gamma(\gamma+3/2)}\cdot z^{\gamma}\cdot e^z.
    \end{align*}
    from which the estimate follows. 
\end{proof}

Finally, we record a useful formula for the derivative of $\Ig$:
\begin{equation}\label{derivativebessel}
    \Ig'(z)=\Igplus(z)+\frac{\gamma}{z}\cdot\Ig(z).
\end{equation}

\section{The heat kernel on an infinite cone}\label{heatkernelsection}
Let $(M,h)$ be a closed Riemannian manifold of dimension $(n-1)\geq 1$. Let $0=\nu_1<\nu_2\leq \nu_3\leq\dots$ denote the eigenvalues of the Laplace-Beltrami operator on $M$, and let $v_k(x)$ be a corresponding orthonormal basis of eigenfunctions. The heat kernel on the cone $C(M)$ is expressed in terms of modified Bessel functions and the eigenfunctions of $M$. In particular, for each $k\geq 1$, let \[\gamma_k:=\sqrt{\frac14(n-2)^2+\nu_k}.\] For $r,s\geq 0$ and $t>0$, let $$\ell(r,s,t):=(rs)^{-n/2}\cdot \frac{rs}{2t}\cdot e^{-\frac{r^2+s^2}{4t}}.$$ Then the heat kernel for $C(M)$ is
\begin{equation}\label{explicitheatkernel}
    p_t((r,x),(s,y)):=\ell(r,s,t)\cdot\sum_{k=1}^{\infty}\Igk\Big(\frac{rs}{2t}\Big)v_k(x)v_k(y).
\end{equation}
See, for instance, the work of Cheeger \cite{cheeger}, Ba\~nuelos-Smits \cite{banuelossmits}, or Huang-Zhang \cite{huang} for details.

In other words, given an initial condition $\phi\in L^2(C(M))$, the solution of the heat equation (\ref{heateqn}) is given by the integrating $\phi$ against $p_t$:
\begin{equation*}
    u(r,x,t)=\int_{C(M)}\phi(s,y)\cdot p_t((r,x),(s,y))dV_g(s,y).
\end{equation*} 

We now set some notation. For $\gamma>0$, let 
\begin{equation}\label{pgamma}
    \Pg(r,s,t):=\ell(r,s,t)\cdot \Ig\Big(\frac{rs}{2t}\Big).
\end{equation}
Abusing notation slightly, we will denote by $dx$ or $dy$ the Riemannian volume density on $(M,h)$. For $k\geq 1$, and an initial condition $\phi\in L^2(C(M))$, let 
\begin{equation*}
    \phi_{\gamma_k}(r):=\int_M\phi(r,y)v_k(y)dy.
\end{equation*}
Define 
\begin{equation}\label{radialsolns}
    w_{\gamma_k}(r,t):=\int_0^{\infty}\Pgk(r,s,t)\cdot \phi_{\gamma_k}(s)\cdot s^{n-1}ds.
\end{equation}
Then, by Equation (\ref{explicitheatkernel}), the solution $u(r,x,t)$ of the heat equation with initial condition $\phi$ is given by 
\begin{equation}\label{fullseries}
    u(r,x,t)=\sum_{k=1}^{\infty}w_{\gamma_k}(r,x,t)v_k(x).
\end{equation}
We prove Theorem \ref{mainthmcones} by finding long-time asymptotic expressions for $w_{\gamma_k}$ for $k=1,2$ on sets of the form $\{r\leq R\cdot t^{1/2}\}$ for $R>0$ and then by finding uniform bounds on the remaining terms in the expansion for $u$.\\
\indent We end the section by stating a simple Gaussian-like upper bound on the heat kernel and deriving a consequence for solutions of the heat equation (\ref{heateqn}).

\begin{lem}\label{gaussianest}
    There exist constants $C_1,C_2>0$ depending on $(M,h)$ such that the following holds for all $t>0$:
    \begin{equation*}
        p_t((r,x),(s,y))\leq C_1\cdot t^{-n/2}\cdot e^{-\frac{(r-s)^2}{C_2 t}}.
    \end{equation*}
\end{lem}
\begin{proof}
    Recall the stronger inequality due to Li (Theorem 1.1 \cite{li}), the $2$-dimensional case of which may be found in \cite{huang}: 
    \[p_t((r,x),(s,y))\leq C_1\cdot t^{-n/2}\cdot e^{-\frac{d((r,x),(s,y))^2}{C_2 t}}\]
    where $d((r,x),(s,y))$ denotes the geodesic distance between $(r,x)$ and $(s,y)$ on $C(M)$. Since this distance is bounded below by $|r-s|$, the result follows. 
\end{proof}

\begin{coro}\label{smallatinfty}
    Suppose that $\phi\in L^2(C(M))$ is a compactly supported initial condition. Let $u$ be the corresponding solution to the heat equation (\ref{heateqn}). Then for any $\epsilon>0$, there exists $R>0$ and $T>0$ such that for all $t\geq T$ and $r\geq R\cdot t^{1/2}$, 
    \begin{equation*}
        |u(r,x,t)|\leq \epsilon \cdot t^{-n/2}.
    \end{equation*}
\end{coro}
\begin{proof}
    Suppose that the support of $\phi$ is contained in $[0,S]\times M$ for some $S>0$. Then for $s\leq S$ and $r\geq 2S$, we have 
    \[e^{-\frac{(r-s)^2}{C_2t}}\leq e^{-\frac{r^2}{4C_2t}},\]
    so Lemma \ref{gaussianest} gives
    \begin{align*}
        u(r,x,t)&=\int_{C(M)}\phi(s,y)\cdot p_t((r,x),(s,y))dV_g(s,y)\\
        &\leq C_1\cdot t^{-n/2}\cdot\int_{C(M)}\phi(s,y)\cdot e^{-\frac{(r-s)^2}{C_2t}}dV_g(s,y)\\
        &\leq C_1\cdot t^{-n/2}\cdot e^{-\frac{r^2}{4C_2t}}\cdot \|\phi\|_{L^1(C(M)).}
    \end{align*}
    For $r\geq R\cdot t^{1/2}>0$, this gives 
    \[u(r,x,t)\leq C_1\cdot t^{-n/2}\cdot e^{-\frac{R^2}{4C_2}}\cdot \|\phi\|_{L^1(C(M)).}\] Taking $R$ sufficiently large and then letting $T>1$ be such that $R\cdot t^{1/2}\geq 2S$ for $t\geq T$ gives the desired inequality. 
\end{proof}

\section{Limiting behavior of individual terms}\label{limitingsection}
We now find long-time asymptotics on the individual terms in the heat kernel (\ref{explicitheatkernel}) on sets of the form $\{r\leq R\cdot t^{1/2}, s\leq S\}$. These asymptotics allow us to determine the limiting behavior of solutions to the heat equation (\ref{heateqn}) on these sets. We continue to use the notation established in Section \ref{heatkernelsection}, but we often suppress the index $k$ in what follows.\\
\indent Let $\gamma>0$, and let $\Pg$ be given by Equation (\ref{pgamma}).
\begin{lem}\label{pgammaasymptotic}
Let $R,S>0$. On the set $\{r\leq R\cdot t^{1/2},s\leq S\}$, the function $\Pg$ satisfies
\begin{equation*}
    \Pg(r,s,t)=\frac{1}{2^{2\gamma+1}\cdot \Gamma(\gamma+1)}\cdot (rs)^{\gamma-n/2+1}\cdot t^{-(\gamma+1)}\cdot e^{-\frac{r^2}{4t}}\cdot\big(1+O(t^{-1})\big)
\end{equation*}
as $t\to\infty$.
\end{lem}
\begin{proof}
    On the set $\{r\leq R\cdot t^{1/2},s\leq S\}$, the quantity $rs/(2t)$ is $O(t^{-1/2})$ as $t\to\infty$, and \[e^{-\frac{s^2}{4t}}=1+O(t^{-1}).\] Using the small-$z$ asymptotic (\ref{besselatzero}), we get 
    \begin{align*}
        \Pg(r,s,t)&=\ell(r,s,t)\cdot \Ig\Big(\frac{rs}{2t}\Big)\\
        &=(rs)^{-n/2}\cdot\frac{rs}{2t}\cdot e^{-\frac{r^2+s^2}{4t}}\cdot \frac{1}{2^{\gamma}\cdot \Gamma(\gamma+1)}\cdot\Big(\frac{rs}{2t}\Big)^{\gamma}\cdot\Big(1+O\Big(\Big(\frac{rs}{2t}\Big)^2\Big)\Big)\\
        &=\frac{1}{2^{2\gamma+1}\cdot \Gamma(\gamma+1)}\cdot (rs)^{\gamma-n/2+1}\cdot t^{-(\gamma+1)}\cdot e^{-\frac{r^2}{4t}}\cdot\big(1+O(t^{-1})\big)
    \end{align*}
\end{proof}

\indent Let $\gamma=\gamma_k$ for some $k$, and define a functional 
\begin{equation*}\Psi_{\gamma}(\phi)=\int_0^{\infty}\phi_{\gamma}(s)\cdot s^{\gamma+n/2}ds=\int_{C(M)}\phi(s,y)\cdot v_k(y)\cdot  s^{\gamma-n/2+1}dV_g(s,y),
\end{equation*}
which is well-defined for the set of functions $\phi\in L^2(C(M))$ with compact support. 
\begin{coro}\label{wasympt}
    Let $R>0$. Let $\phi\in L^2(C(M))$ be a compactly supported initial condition. If $\Psi_{\gamma}(\phi)\neq 0$, then on the set $\{r\leq R\cdot t^{1/2}\}$, the function $w_{\gamma}(r,t)$ defined by (\ref{radialsolns}) satisfies 
    \begin{equation*}
        w_{\gamma}(r,t)=\frac{\Psi_{\gamma}(\phi)}{2^{2\gamma+1}\cdot\Gamma(\gamma+1)}\cdot r^{\gamma-n/2+1}\cdot t^{-(\gamma+1)}\cdot e^{-\frac{r^2}{4t}}\cdot\big(1+O(t^{-1})\big).
    \end{equation*}
    as $t\to\infty$. If $\Psi_{\gamma}(\phi)=0$, then on the same set, we have \[w_{\gamma}(r,t)=O(r^{\gamma-n/2+1}\cdot t^{-(\gamma+2)}).\]
\end{coro}
\begin{proof}
    Suppose that the support of $\phi$ is contained in $[0,S]\times M$ for some $S>0$. The result follows from applying Lemma \ref{pgammaasymptotic} with this value of $S$.
\end{proof}

We also need to understand the limiting value of the radial derivative of $w_{\gamma}$. To do so, we again study the function $\Pg$. Using the formula (\ref{derivativebessel}) of the modified Bessel function, we have
\begin{equation}\label{pderiv}
    \partial_r\Pg(r,s,t)=\Big[\Big(\gamma-\frac{n-2}{2}\Big)\cdot\frac1r-\frac{r}{2t}\Big]\cdot\Pg(r,s,t)+\frac{s}{2t}\cdot P_{\gamma+1}(r,s,t).
\end{equation}
Furthermore, using the relation (\ref{quotientofbessels}), on sets of the form $\{s\leq S\}$, we have 
\begin{equation*}
    \partial_r\Pg(r,s,t)=\Big[\Big(\gamma-\frac{n-2}{2}\Big)\cdot\frac1r-\frac{r}{2t}+r\cdot O(t^{-2})\Big]\cdot\Pg(r,s,t)
\end{equation*}
The following is proved similarly to Corollary \ref{wasympt}
\begin{lem}\label{wderivative}
    Let $R>0$, and suppose that the initial condition $\phi\in L^2(C(M))$ is compactly supported. If $\Psi_{\gamma}(\phi)\neq 0$, then on $\{r\leq R\cdot t^{1/2}\}$, we have
    \begin{align*}
        \partial_rw_{\gamma}(r,t)=\Big[\Big(\gamma-\frac{n-2}{2}\Big)&\cdot\frac1r-\frac{r}{2t}+r\cdot O(t^{-2})\Big]\\&\cdot\frac{\Psi_{\gamma}(\phi)}{2^{2\gamma+1}\cdot\Gamma(\gamma+1)}\cdot r^{\gamma-n/2+1}\cdot t^{-(\gamma+1)}\cdot e^{-\frac{r^2}{4t}}\cdot\big(1+O(t^{-1})\big)
    \end{align*}
    as $t\to\infty$. If $\Psi_{\gamma}(\phi)=0$, then on the same set, we have 
    \[\partial_rw_{\gamma}(r,t)=O(r^{\gamma-n/2}\cdot t^{-(\gamma+2)}).\]
\end{lem}

\section{Tail estimates}\label{tailsection}
The asymptotics provided in Section \ref{limitingsection} include error terms that depend on $\gamma$ and the initial condition in non-trivial ways, so we cannot easily apply these estimates to the entire series defining the heat kernel. Here we make weaker estimates that show that the tail of the series and its radial derivative is negligible on the sets we consider. We begin by recalling some well-known estimates on the eigenvalues and eigenfunctions of the fiber $(M,h)$:
\begin{lem}\label{evalestimates}
    There exists a constant $C>0$ depending only on $(M,h)$ such that each of the following estimates holds:
    \begin{equation}\label{weyl}
        \frac1C\cdot k^{\frac{2}{n-1}}\leq \nu_k\leq C\cdot k^{\frac{2}{n-1}}\;\;\text{for all}\;\;k\geq 2
    \end{equation}
    \begin{equation}\label{gammaweyl}
        \frac1C\cdot k^{\frac{1}{n-1}}\leq \gamma_k\leq C\cdot k^{\frac{1}{n-1}}\;\;\text{for all}\;\;k\geq 1
    \end{equation}
    \begin{equation}\label{linfty}
        \|v_k\|_{L^{\infty}(M)}\leq C\cdot k^{\frac12\cdot \frac{n-2}{n-1}}\;\;\text{for all}\;\;k\geq 1.
    \end{equation}
\end{lem}
\begin{proof}
    Equation (\ref{weyl}) follows immediately from Weyl's law (see Theorem 3.3.4 of \cite{levitin}). Equation (\ref{gammaweyl}) follows from Equation (\ref{weyl}) and the fact that $\gamma_1>0$. The eigenfunction bound follows from the semiclassical estimate $\|v_k\|_{L^{\infty}(M)}\leq C\cdot \nu_k^{\frac{n-2}{4}}$ (see Corollary 5.1.2 of \cite{sogge}) combined with Equation (\ref{weyl}).
\end{proof}

To bound the tails, we find upper bounds on the absolute value of each term in the series (\ref{fullseries}). The following lemma allows us to sum these estimates. 

\begin{lem}\label{sumconverges}
    Let $z\geq 0$ and $\alpha,\beta,c,C>0$. Then
    \[\sum_{k=1}^{\infty}\frac{k^{\beta}\cdot z^{Ck^{\alpha}}}{\Gamma(ck^{\alpha})}<\infty.\]
\end{lem}
\begin{proof}
    Let \[b_k=\Big(\frac{c^{\alpha}}{e}\Big)^{ck^{\alpha}}\cdot k^{\alpha ck^{\alpha}}.\]
    Then by Stirling's approximation formula, there exists $c_1>0$ such that
    \begin{equation}
        \Gamma(ck^{\alpha})\geq c_1\cdot k^{-\alpha/2}\cdot b_k.
    \end{equation}
    Noting that 
    \[\ln(b_k)=ck^{\alpha}\ln\Big(\frac{c^{\alpha}}{e}\Big)+\alpha ck^{\alpha}\ln(k)\geq c_2k^{\alpha}\ln(k)\] for $c_2$ sufficiently small and $k$ sufficiently large, we have \[b_k\geq e^{c_2k^{\alpha}\ln(k)}=k^{c_2k^{\alpha}}.\] Thus, for $k$ sufficiently large, 
    \[\frac{k^{\beta}\cdot z^{Ck^{\alpha}}}{\Gamma(ck^{\alpha})}\leq \frac{1}{c_1}\cdot z^{Ck^{\alpha}}\cdot k^{\beta+\alpha/2-c_2k^{\alpha}}.\] Now if we further suppose that $k\geq z^{2C/c_2}$, then $z^{Ck^{\alpha}}\leq k^{c_2k^{\alpha}/2}$, so we get \[\frac{k^{\beta}\cdot z^{Ck^{\alpha}}}{\Gamma(ck^{\alpha})}\leq \frac{1}{c_1}\cdot k^{\beta+\alpha/2-c_2k^{\alpha}/2}.\] Finally, for $k$ sufficiently large, $\beta-\alpha/2-c_2k^{\alpha}/2<-2$, so for such $k$, we have that each term in the series is bounded above by $k^{-2}/c_1$, so the series converges. 
\end{proof}

Applying Lemma \ref{estimatebessel}, we have 
\begin{equation}\label{Pgupperbound}
    \Pg(r,s,t)\leq \frac{1}{\Gamma(\gamma)}\cdot (rs)^{-n/2}\cdot e^{-\frac{(r-s)^2}{4t}}\cdot \Big(\frac{rs}{2t}\Big)^{\gamma+1}
\end{equation}
for all $r,s\geq 0$ and $t>0$.

We can now bound the tail of the series (\ref{fullseries}). Let $K$ be some positive integer, and define \[U_K(r,x,t)=\sum_{k=K}^{\infty}w_{\gamma_k}(r,x,t)\cdot v_k(x).\]
\begin{prop}\label{normaltail}
    Suppose that the initial condition $\phi\in L^2(C(M))$ has compact support. Let $R>0$. Then there exists $C_1>0$ such that on $\{r\leq R\cdot t^{1/2}, t\geq 1\}$, 
    \begin{equation*}
        |U_K(r,x,t)|\leq C_1\cdot r^{\gamma_K-n/2+1}\cdot t^{-(\gamma_K+1)}.
    \end{equation*}
\end{prop}
\begin{proof}
    We first find an upper bound on each $w_{\gamma}=w_{\gamma_k}$ using (\ref{Pgupperbound}). Suppose that the support of $\phi$ is contained in $[0,S]\times M$ for some $S>0$. Then by the definition of $w_{\gamma}$, we have 
    \begin{align*}
        |w_{\gamma}(r,t)|&\leq \frac{1}{\Gamma(\gamma)}\cdot (rS)^{\gamma-n/2+1}\cdot t^{-(\gamma+1)}\cdot\int_0^S|\phi_{\gamma}(s)|\cdot s^{n-1}ds\\&\leq \frac{S^{n/2}\|\phi\|_{L^2(C(M))}}{\sqrt{n}}\cdot \frac{(rS)^{\gamma-n/2+1}}{\Gamma(\gamma)}\cdot t^{-(\gamma+1)}.
    \end{align*}
    Let $C_1=\frac{1}{\sqrt{n}}\cdot S^{\gamma_K+1}\cdot \|\phi\|_{L^2(C(M))}$. Suppose without loss of generality that $R,S>1$. Summing over $k$ and taking $t\geq 1$ gives, by applying the estimate (\ref{linfty}),
    \begin{align*}
        |U_K(r,x,t)|&\leq C_1\cdot r^{\gamma_K-n/2+1}\cdot t^{-(\gamma_K+1)}\cdot\sum_{k=K}^{\infty}\frac{(rS)^{\gamma_k-\gamma_K}}{\Gamma(\gamma_k)}\cdot t^{\gamma_K-\gamma_k}\cdot \|v_k\|_{L^{\infty}(C(M))}\\
        &\leq C_1\cdot r^{\gamma_K-n/2+1}\cdot t^{-(\gamma_K+1)}\cdot \sum_{k=K}^{\infty}\frac{(RS)^{\gamma_k-\gamma_K}}{\Gamma(\gamma_k)}\cdot t^{(\gamma_K-\gamma_k)/2}\cdot \|v_k\|_{L^{\infty}(C(M))}\\
        &\leq C_1\cdot r^{\gamma_K-n/2+1}\cdot t^{-(\gamma_K+1)}\cdot \sum_{k=K}^{\infty}\frac{(RS)^{Ck^{\frac{1}{n-1}}}}{\Gamma\big(\frac{1}{C}\cdot k^{\frac{1}{n-1}}\big)}\cdot C\cdot k^{\frac12\cdot \frac{n-2}{n-1}}.
    \end{align*}
    By Lemma \ref{sumconverges}, the series in the last inequality converges, and the result follows. 
\end{proof}

We perform a similar analysis on the radial derivative of $U$. Using the inequality (\ref{besselmono}) and Equation (\ref{pderiv}), we have the following on sets of the form $\{r\leq R\cdot t^{1/2},s\leq S\}$:
\[\partial_r\Pg(r,s,t)\leq \Big(\frac{\gamma}{r}+\frac{R}{2t^{1/2}}+\frac{S}{2t}\Big)P_{\gamma}(r,s,t).\] Thus, there exists $T>0$ depending only on $R$ and $S$ (in particular, not on $\gamma$) such that for $t\geq T$, on the set $\{r\leq R\cdot t^{1/2},s\leq S\}$, we have 
\begin{equation*}
    \partial_r\Pg(r,s,t)\leq \Big(\frac{\gamma}{r}+\frac{R}{t^{1/2}}\Big)\Pg(r,s,t).
\end{equation*}
Then a similar computation to that performed in the proof of Proposition \ref{normaltail} gives
\begin{prop}\label{tailderivative}
    Suppose that the initial condition $\phi\in L^2(C(M))$ has compact support. Let $R>0$. Then there exists $C_1>0$ and $T>0$ such that on $\{r\leq R\cdot t^{1/2},t\geq T\}$, 
    \begin{equation*}
        |\partial_rU_K(r,x,t)|\leq C_1\cdot r^{\gamma-n/2}\cdot t^{-(\gamma_K+1)}.
    \end{equation*}
\end{prop}

\section{Proof of Theorem \ref{mainthmcones}}\label{mainproofsection}
We are now set to prove Theorem \ref{mainthmcones}. We suppose throughout the section that $u$ is a solution to the heat equation (\ref{heateqn}) on $C(M)$ with an initial condition $\phi$ satisfying Hypotheses \ref{phihyp} and \ref{hyp2}. Let $p$ denote the cone point on $C(M)$. Recall the definition of $H(t)$ from Section \ref{intro}.\\
\indent As mentioned in Section \ref{intro}, we find even more precise information on the location of $H(t)$ than is stated in Theorem \ref{mainthmcones}. For each $k\geq 1$, let 
\begin{equation*}
    G_k(x):=\frac{\Psi_{\gamma_k}(\phi)}{2^{2\gamma_k+1}\cdot\Gamma(\gamma_k+1)}\cdot v_k(x).
\end{equation*}
Note that since $v_1$ is a constant function equal to $\text{Vol}(M,h)^{-1/2}$, the function $G_1(x)$ is also constant. Suppose that $\nu_2=...=\nu_{K-1}<\nu_K$ for some $K\geq 3$, and let \[J(x):=\sum_{k=2}^{K-1}G_k(x).\] Note that $J(x)$ not being identically equal to $0$ is equivalent to Hypothesis \ref{hyp2}. Let \[m:=\max_{x\in M}J(x)\;\;\text{and}\;\;R_{\infty}:=\Bigg(\frac{2}{G_1}\cdot\Big(\gamma_2-\frac{n-2}{2}\Big)\cdot m\Bigg)^{\displaystyle1/(n/2-\gamma_2+1)}.\] Let \[\alpha:=\frac{n/2-\gamma_2}{n/2-\gamma_2+1}.\] Finally, for any $\epsilon>0$, let 
\begin{equation}\label{ueps}
    U_{\epsilon}:=\{x\in M\mid J(x)\geq m-\epsilon\}.
\end{equation} Then we have 
\begin{thm}\label{moreprecise}
    Under the same hypotheses as Theorem \ref{mainthmcones}, we have 
    \begin{enumerate}
        \item If $\nu_2\geq 2n$, then there exists $T>0$ such that for all $t\geq T$, \[H(t)=\{p\}.\]
        \item If $\nu_2<2n$ and $J(x)$ is not identically equal to $0$, then for any $\epsilon>0$, there exists $T>0$ such that for all $t\geq T$, 
        \[H(t)\subseteq \Big[(R_{\infty}-\epsilon)\cdot t^{\alpha},(R_{\infty}+\epsilon)\cdot t^{\alpha}\Big]\times U_{\epsilon}.\]
    \end{enumerate}
\end{thm}
By unraveling the definitions and notation, Theorem \ref{mainthmcones} follows immediately from Theorem \ref{moreprecise}.

\indent To prove Theorem \ref{moreprecise}, we first show that $H(t)$ does not tend to infinity more quickly than order $t^{1/2}$.
\begin{lem}\label{specificR}
     There exist $R,T>0$ such that for all $t\geq T$, \[H(t)\subseteq [0,R\cdot t^{1/2}]\times M.\]
\end{lem}
\begin{proof}
    Note that $\gamma_1=\frac{n-2}{2}$, so Hypothesis \ref{phihyp} implies that $\Psi_{\gamma_1}(\phi)>0$, and $\gamma_k-n/2+1>0$ if and only if $k\geq 2$. Thus, by the expansion (\ref{fullseries}), Corollary \ref{wasympt}, and Proposition \ref{normaltail}, we have 
    \[u(p,t)=G_1\cdot t^{-n/2}\cdot\big(1+O(t^{-1})\big).\]
    The result then follows from using the $R,T>0$ provided by Corollary \ref{smallatinfty} with $\epsilon=G_1/2$.
\end{proof}

\indent We now combine the results of the preceding two sections to find an asymptotic expansion for $\partial_ru$ on sets of the form $\{r\leq R\cdot t^{1/2}\}$. Let $R>0$ be any positive real number. By Lemma \ref{wderivative} and Proposition \ref{tailderivative}, there exists $T>0$ such that for any $t\geq T$ and $r\leq R\cdot t^{1/2}$, 
\begin{alignat}{2}\label{fullexp}
    \partial_ru&=\sum_{k=1}^{\infty}(\partial_r&&w_{\gamma_k}(r,t))\cdot v_k(x)\nonumber\\
    &=e^{-\frac{r^2}{4t}}\cdot&&\Bigg[-\frac12G_1\cdot r\cdot t^{-\frac{n+2}{2}}\cdot \big(1+O(t^{-1})\big)\\& &&+\Big(\gamma_2-\frac{n-2}{2}-\frac{r^2}{2t}+r^2\cdot O(t^{-2})\Big)\cdot J(x)\cdot r^{\gamma_2-n/2}\cdot t^{-(\gamma_2+1)}\cdot\big(1+O(t^{-1})\big)\nonumber\\& && r^{\gamma_2-n/2}\cdot t^{-(\gamma_2+1)}\cdot O\Big(t^{-\frac{\gamma_K-\gamma_2}{2}}\Big)\Bigg].\nonumber
\end{alignat}
Using this expansion, we can further constrain the long-time location of $H(t)$.
\begin{lem}\label{anyR}
    Let $0<R_1<R_2$. Then there exists $T>0$ such that for all $t\geq T$, the radial derivative $\partial_r u$ is strictly negative on $\big[R_1\cdot t^{1/2},R_2\cdot t^{1/2}\big]\times M$. 
\end{lem}
\begin{proof}
    On this set, the quotient $r^2/(2t)$ is uniformly bounded. By the expansion (\ref{fullexp}), therefore, there exist constants $C_1,C_2>0$ such that for $t$ sufficiently large, \[\partial_ru\leq e^{-\frac{r^2}{4t}}\Big[-C_1\cdot t^{-\frac{n+1}{2}}+C_2\cdot t^{-\gamma_2/2-n/4-1}\Big].\] Since $\gamma_2>\frac{n-2}{2}$, we have \[-\frac{\gamma_2}{2}-\frac{n}{4}-1<-\frac{n+1}{2},\] so we get $\partial_r u<0$ for sufficiently large $t$. 
\end{proof}
Combining Lemma \ref{specificR} with Lemma \ref{anyR} immediately gives
\begin{coro}\label{Htrapped1}
    For any $R>0$, there exists $T>0$ such that for any $t\geq T$, \[H(t)\subseteq \big[0,R\cdot t^{1/2}\big]\times M.\]
\end{coro}
We next restrict the eventual location of $H(t)$ within each fiber. Let $U_{\epsilon}$ be as in (\ref{ueps}).
\begin{lem}\label{Htrapped2}
    For any $\epsilon>0$ and $R>0$, there exists $T>0$ such that for all $t\geq T$, \[H(t)\subseteq \big[0,R\cdot t^{1/2}\big]\times U_{\epsilon}.\]
\end{lem}
\begin{proof}
    Similarly to the computation leading to the expansion (\ref{fullexp}), for $r\leq R\cdot t^{1/2}$ and sufficiently large $t$, we have 
    \begin{align*}
        u=e^{-\frac{r^2}{4t}}\cdot \Big[&G_1\cdot t^{-n/2}\cdot \big(1+O(t^{-1})\big)+J(x)\cdot r^{\gamma_2-n/2+1}\cdot t^{-(\gamma_2+1)}\cdot\big(1+O(t^{-1})\big)\\&r^{\gamma_2-n/2}\cdot t^{-(\gamma_2+1)}\cdot O(t^{-\delta})\Big]
    \end{align*}
    for some $\delta>0$. The result follows directly from this expansion and Corollary \ref{Htrapped1}. 
\end{proof}

We are now set to prove Theorem \ref{moreprecise}. 
\begin{proof}[Proof of Theorem \ref{moreprecise}]
    Let $R>0$. If $\nu_2\geq 2n$, then $\gamma_2\geq \frac{n}{2}+1$. The expansion (\ref{fullexp}) implies that there exist constants $C_1,C_2>0$ such that for $t$ sufficiently large and $r\leq R\cdot t^{1/2}$, 
    \begin{align*}
        \partial_ru&\leq e^{-\frac{r^2}{4t}}\cdot\Big[-C_1\cdot r\cdot t^{-\frac{n+2}{2}}+C_2\cdot r^{\gamma_2-n/2}\cdot t^{-(\gamma_2+1)}\Big]\\
        &=e^{-\frac{r^2}{4t}}\cdot r\cdot t^{-\frac{n+2}{2}}\cdot\Big[-C_1+C_2\cdot r^{\gamma_2-n/2-1}\cdot t^{n/2-\gamma_2}\Big].
    \end{align*}
    Since $\gamma_2-\frac{n}{2}-1\geq 0$, if $0<r\leq R\cdot t^{1/2}$, then we get 
    \[\partial_ru\leq e^{-\frac{r^2}{4t}}\cdot r\cdot t^{-\frac{n+2}{2}}\cdot \Big[-C_1+C_2\cdot R\cdot t^{n/4-\gamma_2/2-1/2}\Big].\] Since the exponent on $t$ in this last expression is negative, we get $\partial_ru<0$ for all $0<r<R\cdot t^{1/2}$ and sufficiently large $t$, and we obtain the first case of the theorem by Corollary \ref{Htrapped1}.\\
    \indent Now suppose that $\nu_2<2n$, so $\gamma_2<\frac{n}{2}+1$. Let $\epsilon>0$. Let $\delta_1,\delta_2,\delta_3>0$ be small (to be specified later). By the expansion (\ref{fullexp}), we have the following for $t$ sufficiently large, $R>0$ sufficiently small, $r\leq R\cdot t^{1/2}$, and $x\in U_{\delta_2}$:
    \begin{alignat*}{2}
        \partial_ru&\geq e^{-\frac{r^2}{4t}}\cdot &&\Big[-\frac12G_1\cdot r\cdot t^{-\frac{n+2}{2}}\cdot (1+\delta_1)\\& &&+\Big(\gamma_2-\frac{n-2}{2}-\delta_2\Big)\cdot (m-\delta_2)\cdot (1-\delta_2)\cdot r^{\gamma_2-n/2}\cdot t^{-(\gamma_2+1)}\\& &&-\delta_3\cdot r^{\gamma_2-n/2}\cdot t^{-(\gamma_2+1)}\Big]\\
        &=e^{-\frac{r^2}{4t}}\cdot &&\;r^{\gamma_2-n/2} \cdot \Big[-\frac12 G_1\cdot r^{n/2-\gamma_2+1}\cdot t^{-\frac{n+2}{2}}\cdot (1+\delta_1)\\& &&+\bigg(\Big(\gamma_2-\frac{n-2}{2}-\delta_2\Big)\cdot(m-\delta_2)\cdot(1-\delta_2)-\delta_3\bigg)\cdot t^{-(\gamma_2+1)}\Big].
    \end{alignat*}
    Now supposing further that $r\leq (R_{\infty}-\epsilon)\cdot t^{\alpha}$ gives 
    \begin{align*}
        \partial_r u\geq e^{-\frac{r^2}{4t}}\cdot r^{\gamma_2-n/2}\cdot t^{-(\gamma_2+1)}\cdot\Big[&-\frac12 G_1\cdot \big(R_{\infty}-\epsilon)^{n/2-\gamma_2+1}\cdot (1+\delta_1)\\&+\Big(\gamma_2-\frac{n-2}{2}-\delta_2\Big)\cdot (m-\delta_2)\cdot (1-\delta_2)-\delta_3\Big].
    \end{align*}
    By the definition of $R_{\infty}$, taking $\epsilon$, $\delta_1$, $\delta_2$, and $\delta_3$ equal to $0$ makes this last term equal to $0$. By continuity, we see that choosing $\delta_1$ then $\delta_2$ then $\delta_3$ each to be sufficiently small with respect to $\epsilon$ and to each other gives $\partial_ru>0$ on $r\leq (R_{\infty}-\epsilon)\cdot t^{\alpha}$ with $t$ sufficiently large. A similar computation gives the opposite inequality on $(R_{\infty}+\epsilon)\cdot t^{\alpha}\leq r\leq R\cdot t^{1/2}$. If the chosen $\delta_2$ exceeds $\epsilon$, then replace it by $\epsilon$ to additionally ensure that $x\in U_{\epsilon}$ for each $(r,x)\in H(t)$ for sufficiently large $t$. By Lemma \ref{Htrapped2}, the theorem follows. 
\end{proof}


\end{document}